\documentclass{amsart}

\usepackage{amsthm}
\usepackage{amssymb}
\usepackage{amsmath}
\usepackage{amscd}
\usepackage{color}

\newcommand{\La}{\Lambda}
\newcommand{\As}{\mathcal{A}_{\mathrm{s}}}
\newcommand{\Aw}{\mathcal{A}_{\mathrm{w}}}
\newcommand{\A}{\mathcal{A}}

\newcommand{\ds}{\mathrm{d}_{\mathrm{s}}}
\newcommand{\dw}{\mathrm{d}_{\mathrm{w}}}
\newcommand{\Ec}{\mathcal{E}}

\newcommand{\Xw}{X_{\mathrm{w}}}

\newcommand{\db}{\mathrm{d}_{\bullet}}

\newcommand{\dd}{\,d}

\newcommand{\bg}{\begin{equation}}
\newcommand{\ed}{\end{equation}}
\newcommand{\bga}{\begin{eqnarray}}
\newcommand{\eda}{\end{eqnarray}}

\def\cbdu{\par{\raggedleft$\Box$\par}}

\newtheorem {Theorem}  {Theorem}

\numberwithin{Theorem}{section}

\newtheorem {Lemma}[Theorem]  {Lemma}

\theoremstyle{definition}
\newtheorem{Definition}[Theorem]{Definition}
\theoremstyle{remark}

\expandafter\chardef\csname pre amssym.def
at\endcsname=\the\catcode`\@ \catcode`\@=11
\def\undefine#1{\let#1\undefined}
\def\newsymbol#1#2#3#4#5{\let\next@\relax
 \ifnum#2=\@ne\let\next@\msafam@\else
 \ifnum#2=\tw@\let\next@\msbfam@\fi\fi
 \mathchardef#1="#3\next@#4#5}
\def\mathhexbox@#1#2#3{\relax
 \ifmmode\mathpalette{}{\m@th\mathchar"#1#2#3}%
 \else\leavevmode\hbox{$\m@th\mathchar"#1#2#3$}\fi}
\def\hexnumber@#1{\ifcase#1 0\or 1\or 2\or 3\or 4\or 5\or 6\or 7\or 8\or
 9\or A\or B\or C\or D\or E\or F\fi}

\font\teneufm=eufm10 \font\seveneufm=eufm7 \font\fiveeufm=eufm5
\newfam\eufmfam
\textfont\eufmfam=\teneufm \scriptfont\eufmfam=\seveneufm
\scriptscriptfont\eufmfam=\fiveeufm

\catcode`\@=\csname pre amssym.def at\endcsname

\newcounter{remark}
\usepackage{color}

\newcommand{\e}{\epsilon}

\renewcommand{\l}{\lambda}

\renewcommand{\th}{\theta}

\newcommand{\R}{\mathbf{R}}

\def  \R   {{\mathbb R}}
\def  \Z   {{\mathbb Z}}

\def  \T   {{\mathbb T}}

\def  \12  {{\frac{1}{2}}}

\def  \l   {\langle}

\def\build#1_#2^#3{\mathrel{\mathop{\nuern 0pt#1}\limits_{#2}^{#3}}}

\begin{document}

\title[Global Attractor of SQG in $L^2$]{The existence of a global attractor for the forced critical surface quasi-geostrophic Equation in $L^2$}

\author[Alexey Cheskidov]{ Alexey Cheskidov }
\address{Department of Mathematics, University of Illinois, Chicago, IL 60607,USA}
\email{acheskid@uic.edu}

\author[Mimi Dai]{ Mimi Dai}
\address{Department of Mathematics, University of Illinois, Chicago, IL 60607,USA}
\email{mdai@uic.edu}

\thanks{The work of Alexey Cheskidov was partially supported by NSF Grant DMS-1108864}

\begin{abstract}
We prove that the critical surface quasi-geostrophic equation driven by a force $f$ possesses a compact global attractor in $L^2(\mathbb T^2)$ provided
$f\in L^p(\mathbb T^2)$ for some $p>2$. First, the De Giorgi method is
used to obtain uniform $L^\infty$ estimates on viscosity solutions. Even though
this does not provide  a compact absorbing set, the existence
of a compact global attractor follows from the continuity of solutions, which is
obtained by estimating the energy flux using the Littlewood-Paley decomposition.

\end{abstract}

\maketitle

\section{Introduction}
We consider the two dimensional critical surface quasi-feostrophic (SQG) equation

\begin{equation}\begin{split}\label{QG}
\frac{\partial\theta}{\partial t}+u\cdot\nabla \theta+\nu\Lambda\theta =f,\\
u=R^\perp\theta,
\end{split}
\end{equation}
on the torus $\mathbb T^2=[0,L]^2$, where $\nu>0$, $\Lambda=\sqrt{-\Delta}$ is the Zygmund operator, and
\bg\notag
R^\perp\theta=\Lambda^{-1}(-\partial_2\theta,\partial_1\theta).
\ed
The scalar function $\theta$ represents the potential temperature and the vector function $u$ represents the fluid velocity. 
The initial data $\th(0) \in L^2(\mathbb{T}^2)$ and the force $f \in L^p(\mathbb{T}^2)$ for some $p>2$ are assumed to have zero average.
Equation (\ref{QG}) describes the evolution of the surface temperature field in a rapidly rotating and stably stratified fluid with potential velocity \cite{CMT}. As pointed out in \cite{CMT}, this equation attracts interest of  scientists and mathematicians due to two major reasons. First, it is a fundamental model for the actual geophysical flows with applications in atmosphere and oceanography study. Second, from the mathematical point of view, the behavior of strongly nonlinear solutions to (\ref{QG}) with $\nu=0, f=0$ in 2D and the behavior of potentially singular solutions to the Euler's equation in 3D are strikingly analogous which has been justified both analytically and numerically. For literature the readers are refereed to \cite{CCW, CMT, CW, Pe} and the references therein. 

Equation (\ref{QG}) is usually referred as the critical SQG \cite{CMT}, since the highest controlled norm is scaling invariant. However, it is not known whether a dramatic change in the behavior of solutions occurs when the dissipation power crosses $1$. The global regularity problem of the critical SQG equation has been very challenging due to the balance of the nonlinear term and the dissipative term. In the unforced case this problem has been resolved by Kieslev, Nazarov and Volberg \cite{KNV}, Caffarelli and Vasseur \cite{CaV}, Kieslev and Nazarov \cite{KN09} and Constantin and Vicol \cite{CV} independently, using different sophisticated methods. The main ingredient of the proof of global regularity in \cite{KNV} relies on constructing a special family of Lipschitz moduli of continuity that are preserved by the dissipative evolution. The proof of \cite{CaV} applies the ideas of De Giorgi iteration method to the nonlocal parabolic equation. As a first step, the authors make use of the interplay between $|\Lambda ^{\frac 12}\theta|$ and $|\theta|$ to prove that a weak solution in $L^2$ is bounded in $L^\infty$. The second step uses a more delicate analysis to show that such a solution is H\"older continuous. Later, to find a bridge between the proofs of \cite{KNV} and \cite{CaV}, Kieslev and Nazarov \cite{KN09} reproved the global regularity using a completely different method which applied elementary tools to control the H\"older norms by choosing a suitable family of test functions (Hardy molecules). Very recently, Constantin and 
Vicol \cite{CV} proposed another proof of the global regularity which provided a transparent way to see that the dissipation is dominating the nonlinear term in the critical SQG equation. The main tool is a nonlinear maximum principle which introduces nonlinear lower bounds for the linear nonlocal operator $\Lambda$.
Applying the method introduced in \cite{KNV}, global well-posdeness and a decay estimate are established for the critical dissipative SQG equation in the whole space in \cite{DD}, global regularity is obtained for the critical SQG equation with smooth forcing term and initial data in \cite{FPV} and for the equation with a linear dispersive force and sufficiently smooth initial data in \cite{KN10}, 
and H\"older continuity is obtained for critical linear drift-diffusion equations in \cite{SV}.

There have also been a few results regarding the long time behavior of the solutions to the critical SQG equation \cite{CTVan, CTV, Don, SS03, SS05}. In the unforced case, the decay rate and lower bound of decay rates have been obtained for mild solutions in \cite{SS05}; the decay estimates have been established for regular solutions in \cite{Don}. In the forced case, decay rate has been obtained for weak solutions, provided the force is time dependent and satisfies certain decay assumptions \cite{SS03}. In the forced case with a special class of time independent force, the long time average behavior of viscosity solutions has been addressed and the absence of anomalous dissipation is obtained by Constantin, Tarfulea, and Vicol in \cite{CTVan}. Recently the authors also studied the long time dynamics of regular solutions of the forced critical SQG based on their new proof of regularity \cite{CTV}. This proof is described to be dynamic in the sense that the H\"older norm of a solution is shown to be dependent
only on the force for
large enough time. With the assumption that the time independent force $f\in L^\infty(\mathbb T^2)\cap H^1(\mathbb T^2)$ and the initial data in $H^1(\mathbb T^2)$, the authors proved the existence of a compact attractor. It is a global attractor in the
classical sense in $H^s(\mathbb T^2)$ for $s\in(1,3/2)$, and it attracts all the points (but not bounded sets) in $H^1(\mathbb T^2)$. 
Moreover, the authors proved that the attractor has a finite box-counting dimension.

In the present paper we prove that the critical SQG equation (\ref{QG}) possesses a global attractor in
$L^2(\mathbb T^2)$, provided the force $f$ is solely in $L^p(\mathbb T^2)$ for $p>2$. As the first step, it is established that for any initial data in $L^2(\mathbb T^2)$ a weak (viscosity) solution is bounded in $L^\infty$ on any interval $[t_0, \infty)$, $t_0>0$. The main tool is an application of the De Giorgi iteration method to the forced critical SQG as it was done by Caffarelli and Vasseur in \cite{CaV} in the unforced case. This is the only 
part that requires the force to be in $L^p$ for some $p>2$. Second, in the spirit of Cheskidov, Constantin, Friedlander, and Shvydkoy 
result \cite{CCFS} on Onsager's conjecture, the Littlewood-Paley decomposition technique is used to show that bounded weak solutions have zero energy flux and hence satisfy the energy equality. The energy equality immediately implies the continuity of weak solutions in $L^2(\mathbb T^2)$. In the third step, we follow an abstract framework of evolutionary systems introduced by Cheskidov and Foias \cite{CF} to show the existence of a weak
global attractor. Finally, with all the above ingredients at hand, we are able to apply a result established in \cite{C5} by Cheskidov to prove that the
weak global attractor is in fact a strongly compact strong global attractor. Namely, we prove the following.


\begin{Theorem}\label{thm:attractor-intro}
Assume $f\in L^p(\mathbb T^2)$ with $p>2$. Then the critical SQG equation (\ref{QG}) possesses a compact global attractor $\A$ in
$L^2(\mathbb T^2)$,
\[
\A=\{ \theta_0: \ \theta_0=\theta(0) \mbox{ for some bounded complete (ancient) viscosity solution } \theta(t) \}.
\] 
In addition, for any bounded set $B \subset L^2(\mathbb T^2)$, $\epsilon >0$, and $T>0$, there exists $t_0$,
such that for any $t^*>t_0$,  every viscosity solution $\theta(t)$ with $\th(0) \in B$
satisfies
\[
\|\theta(t)- v(t)\|_{L^2} < \epsilon, \qquad \forall t\in [t^*,t^*+T],
\]
for some complete trajectory $v(t)$ on the global attractor ($v(t) \in \A \ \forall \ t \in(-\infty, \infty)$).

\end{Theorem}

\medskip

Note that the regularity of solutions is not needed for our approach, and is in fact an open problem. Therefore, even with stronger assumptions
on the force $f\in L^\infty(\mathbb T^2)\cap H^1(\mathbb T^2)$, we only know that the $L^2$-global attractor contains Constantin-Vicol $H^1$-attractor.
 It is an open question whether they coincide.

Finally, we would like to remark that classical arguments require a compact absorbing set in order to prove the existence of a global attractor.  Nevertheless, in \cite{B}, Ball proved the existence of the attractor for the 3D NSE
under the assumption that all the solutions are continuous in $L^2$, which is an open problem.
Later it was shown that the continuity of solutions on the weak global attractor
is enough to conclude that the attractor is strong in the case of the 3D NSE \cite{CF,R}, as well as general evolutionary systems \cite{C5},  but there was no example where
such a method could be applied, except cases where the force is small and the attractor is a fixed point.
To the best of our knowledge, the critical SQG is the first example where
 the existence of the attractor can be proved using such an argument, but the classical approach does not work.
 Indeed, the $L^\infty$-absorbing set is not compact in $L^2$ and hence classical
compactness arguments cannot be applied in this situation.

The paper is organized as follows: Section \ref{sec:infty} is devoted to proving that initial data in $L^2$ produces weak (viscosity) solutions bounded in $L^\infty$; Section \ref{sec:con} is devoted to proving that bounded weak solutions are actually continuous in $L^2$; in Section \ref{sec:att} we prove that a strong global attractor exists in $L^2$.

\subsection{Notation}
\label{sec:notation}
We denote by $A\lesssim B$ an estimate of the form $A\leq C B$ with
some absolute constant $C$, and by $A\sim B$ an estimate of the form $C_1
B\leq A\leq C_2 B$ with some absolute constants $C_1$, $C_2$.

To simplify the notations we denote $\|\cdot\|_p=\|\cdot\|_{L^p}$, and $(\cdot, \cdot)$ stands for the $L^2$-inner product.

\section{$L^\infty$ estimate}
\label{sec:infty}

The goal of this section is to show that viscosity solutions to \eqref{QG} are uniformly bounded in  $L^\infty$ provided the force $f$ is in $L^p$ for some $p>2$.

\begin{Definition}
A weak solution to \eqref{QG} is a function $\th \in C_{\mathrm{w}}([0,T];L^2(\mathbb{T}^2))$ with zero spatial average that satisfies \eqref{QG} in a distributional sense.
That is, for any $\phi\in C_0^\infty(\mathbb T^2\times(0,T))$,
\begin{equation}\notag
-\int_0^T(\theta, \phi_t)dt-\int_0^T(u\theta, \nabla\phi)dt+\nu\int_0^T(\Lambda^{\frac{1}{2}}\theta, \Lambda^{\frac{1}{2}}\phi)dt
=(\theta_0, \phi(x,0))+\int_0^T(f, \phi)dt.
\end{equation}
\end{Definition}

A weak solution $\th(t)$ on $[0,T]$ is called a viscosity solution if there exist sequences $\e_n \to 0$ and $\th_n(t)$ satisfying
\begin{equation}\begin{split}\label{VQG}
\frac{\partial\theta_n}{\partial t}+u_n\cdot\nabla \theta_n+\nu\Lambda\theta_n + \e_n \Delta \th_n=f,\\
u_n=R^\perp\theta_n,
\end{split}
\end{equation}
such that $\th_n \to \theta$ in $C_\mathrm{w}([0,T];L^2)$. Standard arguments imply that for any initial data $\th_0 \in L^2$ there exists a viscosity solution $\th(t)$ of \eqref{QG} on $[0,\infty)$ with $\th(0)=\th_0$ (see \cite{CC}, for example).

In the case of zero force, Caffarelli and Vasseur derived a level set energy inequality using a harmonic extension \cite{CaV}. Even though the force does not present any problems, we sketch a different proof here for completeness.

\begin{Lemma}
\label{le:existence}
Let $\th(t)$ be a viscosity solution to \eqref{QG} on $[0,T]$ with $\th(0) \in L^2$. Then
for every $\lambda \in \mathbb{R}$ it satisfies the level set energy inequality
\begin{equation}\label{truncated}
\frac{1}{2}\|\tilde\theta_\lambda(t_2)\|_2^2+\nu\int_{t_1}^{t_2}\|\Lambda ^{\frac{1}{2}}\tilde\theta_\lambda\|_2^2 \, dt
\leq \frac{1}{2}\|\tilde\theta_\lambda(t_1)\|_2^2+\int_{t_1}^{t_2}\int_{\mathbb T^2}f\tilde\theta_\lambda \, dxdt,
\end{equation}
for all $t_2\in[t_1, T]$ and a.e. $t_1\in [0,T]$. Here $\tilde\theta_\lambda=(\theta-\lambda)_+$ or $\tilde\theta_\lambda=(\theta_\lambda+\lambda)_{-}$.
\end{Lemma}

\begin{proof}

We only show a priori estimates. It is clear how to pass to the limit in \eqref{VQG} as $\e \to 0$.
Denote $\varphi(\theta)=(\theta-\lambda)_+$. Note that $\varphi$ is Lipschitz and 
\begin{equation}\notag
\varphi'(\theta)\varphi(\theta)=\varphi(\theta).
\end{equation}
Multiplying the first equation of (\ref{QG}) by $\varphi'(\theta)\varphi(\theta)$ and integrating over $\mathbb T^2$ yields
\begin{equation}\label{energy1}
\begin{split}
\frac{1}{2}\frac{d}{dt}\int_{\mathbb T^2}\varphi^2(\theta)dx+\int_{\mathbb T^2}\nabla\cdot\left(\frac{1}{2}\varphi^2(\theta)u\right)\, dx\\
+\nu\int_{\mathbb T^2}\Lambda\theta\varphi(\theta)\, dx
=\int_{\mathbb T^2}f\varphi(\theta)\, dx.
\end{split}
\end{equation}
Note that 
\begin{equation}\notag
\Lambda\theta\varphi(\theta)=0,  \qquad \mbox {on the set} \qquad \left\{ x\in\mathbb T^2: \theta(x)\leq \lambda\right\}.
\end{equation}
On the other hand, on the set $\left\{ x\in\mathbb T^2: \theta(x)> \lambda\right\}$, we have (see \cite{CC})
\begin{equation}\notag
\begin{split}
\Lambda\theta(x)&=\frac{1}{2\pi}\sum_{j\in \mathbb Z^2} P.V. \int_{\mathbb T^2}\frac{\theta(x)-\theta(x-y)}{|y+Lj|^3}\, dy\\
&=\frac{1}{2\pi}\sum_{j\in \mathbb Z^2}P.V.\int_{\mathbb T^2}\frac{\theta(x)-\lambda-[\theta(x-y)-\lambda]}{|y+Lj|^3}\, dy\\
&\geq \frac{1}{2\pi}\sum_{j\in \mathbb Z^2} P.V.\int_{\mathbb T^2}\frac{\varphi(\theta(x))-[\theta(x-y)-\lambda]_+}{|y+Lj|^3}\, dy\\
&=\Lambda\varphi(\theta(x)),
\end{split}
\end{equation}
where $L$ is the length of the periodic box.
Therefore, $\Lambda\theta \varphi(\theta) \geq  \Lambda\varphi(\theta)\varphi(\theta)$ for almost every $x\in\mathbb T^2$, and hence
\[
\int_{\mathbb T^2}\Lambda \theta \varphi(\theta)\, dx \leq  \int_{\mathbb T^2}\left|\Lambda^{\frac{1}{2}}\varphi(\theta)\right|^2\, dx.
\]
 Thus, it follows from (\ref{energy1}) that
\begin{equation}\label{energy2}
\begin{split}
\frac{1}{2}\frac{d}{dt}\int_{\mathbb T^2}\varphi^2(\theta)dx+\int_{\mathbb T^2}\nabla\cdot\left(\frac{1}{2}\varphi^2(\theta)u\right)\, dx\\
+\nu\int_{\mathbb T^2}\left|\Lambda^{\frac{1}{2}}\varphi(\theta)\right|^2\, dx
\leq\int_{\mathbb T^2}f\varphi(\theta)\, dx.
\end{split}
\end{equation}
Since the integral $\int_{\mathbb T^2}\nabla\cdot\left(\frac{1}{2}\varphi^2(\theta)u\right)\, dx=0$,
this gives us the truncated energy inequality (\ref{truncated}). The case of $\varphi(\theta)=(\theta+\lambda)_-$ is similar.

\end{proof}

The proof of the following theorem is similar to the one in \cite{CaV}. We just have to take extra care of the forcing term.

\begin{Lemma}
\label{Linfty}
Let $\theta$ be a viscosity solution of \eqref{QG} on $[0,\infty)$ with $\th(0) \in L^2$ and $f\in L^p(\mathbb T^2)$ for some $p>2$. Then
\begin{equation}\notag
\th \in L^{\infty}({\mathbb T}^2 \times (\varepsilon,\infty) ),
\end{equation}
for every $\varepsilon>0$. More precisely,
\[
\|\theta(t)\|_\infty \lesssim \frac{\|\theta(0)\|_2}{\nu t} + \frac{L^{1-\frac2p}}{\nu} \|f\|_p(1+L^{\frac12}\nu^{-\frac12} t^{-\frac12}), \qquad t>0.
\]
\end{Lemma}
\begin{proof} Take the levels 
\begin{equation}\notag
\lambda_k=M(1-2^{-k})
\end{equation}
for some $M$ to be determined later,
and denote the truncated function 
\[
\theta_k=(\theta-\lambda_k)_+.
\]
Fix $t_0>0$. 
We aim to find a bound on $\|\theta(t_0)\|_{\infty}$. 
Let $T_k=t_0(1-2^{-k})$ and 
the level set of energy as:
\begin{equation}\notag
U_k=\sup_{T_k\leq t \leq t_0}\|\theta_k(t)\|_2^2+2\nu\int_{T_k}^{t_0}\|\Lambda^{\frac{1}{2}}\theta_k(t)\|_2^2 \,dt.
\end{equation}
We take $\tilde\theta=\theta_k$ and $t_1=s \in(T_{k-1},T_k)$, $t_2=t>T_k$ in the truncated energy inequality (\ref{truncated}).
Then taking $t_1=s$, $t_2=t_0$, adding two inequalities, and taking $\sup$ in $t$ gives 
\begin{equation}\notag
U_k\leq 2\|\theta_k(s)\|_2^2+4\int_{T_{k-1}}^{t_0} \int_{\mathbb T^2}\left|f(x) \theta_k(x,\tau)\right|\, dxd\tau,
\end{equation}
for a.a. $s \in (T_{k-1}, T_k)$.
Taking the average in $s$ on $[T_{k-1}, T_k]$ yields
\begin{equation}\label{truncated1}
U_k\leq \frac{2^{k+1}}{t_0}\int_{T_{k-1}}^{t_0}\int_{\mathbb T^2} \theta_k^2(s)\,dxds+4\int_{T_{k-1}}^{t_0}\int_{\mathbb T^2}\left|f(x) \theta_k(x,t)\right|\, dxdt.
\end{equation}

By interpolation between $L^\infty(L^2)$ and $L^2(H^{\frac{1}{2}})$, there
exits an absolute constant $C$, such that for any $t_0 \leq L/\nu$ we have
\begin{equation}\label{interpolation}
\int_{T_{k}}^{t_0}\int_{\mathbb T^2} |\theta_{k}|^3\,dxdt \leq 
 C\left(\frac{1}{2\nu}
+\frac{t_0}{2L}\right) U_k^{3/2} \leq \frac{C}{\nu} U_k^{3/2}.
\end{equation}

Note that
\begin{equation}\notag
\theta_{k-1}\geq 2^{-k}M \qquad \mbox {on} \qquad  \left\{(x,t):\theta_k(x,t)>0\right\},
\end{equation}
which implies
\begin{equation}\notag
1_{\left\{\theta_k>0\right\}}\leq \frac{2^k}{M}\theta_{k-1}.
\end{equation}
Therefore, using (\ref{interpolation}) and the fact that $\theta_k\leq\theta_{k-1}$, we have
\begin{equation}\label{truncated2}
\begin{split}
\frac{2^{k+1}}{t_0}\int_{T_{k-1}}^{t_0}\int_{\mathbb T^2}& \theta_k^2(x,s)\,dxds\\
\leq &\frac{2^{k+1}}{t_0}\int_{T_{k-1}}^{t_0}\int_{\mathbb T^2} \theta_{k-1}^2(x,s)1_{\left\{\theta_k>0\right\}}\,dxds\\
\leq &\frac{2^{2k+1}}{t_0M}\int_{T_{k-1}}^{t_0}\int_{\mathbb T^2} |\theta_{k-1}|^3\,dxds\\
\leq &C\frac{2^{2k+1}}{\nu t_0M}U_{k-1}^{3/2}.
\end{split}
\end{equation}
On the other hand, since $f\in L^p(\mathbb T^2)$ with $p>2$, we obtain, for $p'=\frac{p}{p-1}$,
\begin{equation}\label{truncated3}
\begin{split}
&\int_{T_{k-1}}^{t_0} \int_{\mathbb T^2}\left|f(x) \theta_k(x,t)\right|dxdt\\
\leq &\|f\|_p\int_{T_{k-1}}^{t_0} \left(\int_{\mathbb T^2}|\theta|^{p'}_kdx\right)^{1/{p'}}dt\\
\leq &\|f\|_p\int_{T_{k-1}}^{t_0} \left(\int_{\mathbb T^2}|\theta|^{p'}_{k-1}1^2_{\left\{\theta_k>0\right\}}dx\right)^{1/{p'}}dt\\
\leq &\|f\|_p\frac{2^{{2k/{p'}}}}{M^{2/{p'}}}\int_{T_{k-1}}^{t_0} \left(\int_{\mathbb T^2}|\theta_{k-1}|^{2+p'}dx\right)^{1/{p'}}dt\\
\leq &\|f\|_p\frac{2^{{2k/{p'}}}}{M^{2/{p'}}}\sup_{t\geq T_{k-1}}\left(\int_{\mathbb{T}^2}|\theta_{k-1}|^2\, dx\right)^{(2-p')/{2p'}}\int_{T_{k-1}}^{t_0} \left(\int_{\mathbb T^2}|\theta_{k-1}|^4dx\right)^{1/2}dt\\
\leq &\|f\|_p\frac{2^{{2k/{p'}}}}{\nu M^{2/{p'}}}U_{k-1}^{1+(2-p')/{2p'}}.
\end{split}
\end{equation}
Note that $(2-p')/{2p'}\in(0,1/2]$.
Combining (\ref{truncated1}), (\ref{truncated2}), and (\ref{truncated3}) yields
\begin{equation}\label{iteration}
U_k\leq C\frac{2^{2k+1}}{\nu t_0M}U_{k-1}^{3/2}+C\|f\|_p\frac{2^{{2k/{p'}}}}{\nu M^{2/{p'}}}U_{k-1}^{1+(2-p')/{2p'}},
\end{equation}
with an adimensional  constant $C$ independent of $k$. For a large enough $M$, more precisely for  
\begin{equation} \label{eq:M}
M \sim \frac{U_0^{\frac{1}{2}}}{\nu t_0} + \left(\frac{\|f\|_p}{\nu}\right)^{\frac{p}{2p-2}} U_0^{\frac{p-2}{4p-4}},
\end{equation}
the above nonlinear iteration inequality implies that $U_k$ converges to 0 as $k\to \infty$. Thus $\theta(x,t_0)\leq M$ for almost every $x$.
The same argument applied to $\theta_k=(\theta+\lambda_k)_-$ gives also
a lower bound.

Now due to the energy inequality,
\begin{equation} \label{eq:ei-simple}
U_0 \lesssim \|\th(0)\|_2^2 + \frac{t_0L}{\nu}\|f\|_2^2 \leq
\|\th(0)\|_2^2 + \frac{t_0}{\nu}L^{3-\frac4p}\|f\|_p^2
\end{equation}
Notice that we are in the autonomous case, and hence we can use \eqref{eq:M} to bound $\|\theta (t+t_0)\|_\infty$ in terms of
$\|\theta(t)\|_2$ for all $t>0$.
So applying \eqref{eq:M} on intervals of length $t_0$, we immediately obtain that $\th \in L^{\infty}({\mathbb T}^2 \times (\varepsilon,\infty) )$
for every $\varepsilon>0$ due to the fact that there is an absorbing ball in $L^2$ (see
Section 4). It is also easy to obtain an explicit bound on the whole interval $(0,\infty)$
by first combining the bound \eqref{eq:ei-simple} with \eqref{eq:M}, which gives
\begin{equation} \label{eq:Linf-firstbound}
\|\theta(t)\|_\infty \lesssim \frac{\|\theta(0)\|_2}{\nu t} +\frac{L^{1-\frac2p}}{\nu} \|f\|_p(1+L^{\frac12}\nu^{-\frac12} t^{-\frac12}), \qquad t \leq L/\nu.
\end{equation}

To show that this bound also holds for $t>L/\nu$, we fix $t=T=L/\nu$ in \eqref{eq:Linf-firstbound}
and then shift it by $t-T$ in time obtaining
\begin{equation} \label{eq:Linf-firstbound2}
\|\theta(t)\|_\infty \lesssim \frac{\|\theta(t-T)\|_2}{\nu T} +\frac{L^{1-\frac2p}}{\nu} \|f\|_p, \qquad t\geq T.
\end{equation}
Now note that thanks to the energy inequality,
\[
\begin{split}
\|\th(t-T)\|_2^2 &\lesssim \|\th(0)\|_2^2 e^{-\nu \frac{2\pi}{L}(t-T)} + \frac{L^2}{\nu^2} \|f\|_2^2\\
&\lesssim \|\th(0)\|_2^2 e^{-\nu \frac{2\pi}{L}(t-T)} + \frac{L^{4-\frac4p}}{\nu^2} \|f\|_p^2.
\end{split}
\]
Combining this with \eqref{eq:Linf-firstbound2}, recalling that $T=L/\nu$, and using the fact that $e^{-(t-1)} \leq  1/t$
on $(0,\infty)$,  we arrive at
\[
\begin{split}
\|\theta(t)\|_\infty \lesssim \frac{\|\theta(0)\|_2}{\nu T}e^{-\nu \frac{\pi}{L}(t-T)} +\frac{L^{1-\frac2p}}{\nu} \|f\|_p\\
 \lesssim \frac{\|\theta(0)\|_2}{\nu t} +\frac{L^{1-\frac2p}}{\nu} \|f\|_p
,\qquad t\geq T.
\end{split}
\]


\end{proof}

\section{Continuity in $L^2$}
\label{sec:con}

\subsection{Littlewood-Paley decomposition}
\label{sec:LPD}
The techniques presented in this section rely strongly on the Littlewood-Paley decomposition that we recall here briefly. For a more detailed description on this theory we refer the readers to the books by Bahouri, Chemin and Danchin \cite{BCD} and Grafakos \cite{Gr}.

We denote $\lambda_q=2^q$ for an integer $q$. A nonnegative radial function $\chi\in C_0^\infty(\R^n)$ is chosen such that 
\begin{equation}\notag
\chi(\xi)=
\begin{cases}
1, \ \ \mbox { for } |\xi|\leq\frac{1}{2}\\
0, \ \ \mbox { for } |\xi|\geq 1.
\end{cases}
\end{equation}
Let $\varphi(\xi)=\chi(2^{-q}\xi)-\chi(\xi)$
and
\begin{equation}\notag
\varphi_q(\xi)=
\begin{cases}
\varphi(2^{-q}\xi)  \ \ \ \mbox { for } q\geq 0,\\
\chi(\xi) \ \ \ \mbox { for } q=-1.
\end{cases}
\end{equation}
For a tempered distribution vector field $u$ on the torus $\T^n$ we consider the Littlewood-Paley projection
\begin{equation}
u_q (x):=\Delta_qu = \sum_{k \in \Z^n} \hat{u}_k \varphi_q(k) e^{i k\cdot x}, \quad q \geq -1,
\end{equation}
where $\hat{u}_k$ is the Fourier coefficient of $u$. Then the Littlewood-Paley decomposition
\bg\notag
u=\sum_{q=-1}^\infty u_q
\ed
holds in the sense of distributions. To simplify the notation, we denote
\bg\notag
u_{\leq Q}=\sum_{q=-1}^Qu_q, \qquad u_{> Q}=\sum_{q=Q+1}^\infty u_q.
\ed
We will also use
\[
h_Q(x) := \sum_{k \in \Z^n} \chi(\lambda_Q^{-1}k) e^{i k\cdot x}.
\]

\subsection{Energy equality}

In this section we prove that $\theta\in L^\infty(\mathbb T^2 \times (t_0,\infty))$ for all $t_0>0$ implies that $\th(t)$ satisfies the basic energy equality. Namely,
\begin{Theorem}\label{basic-energy}
Let $\theta(t)$ be a viscosity solution of \eqref{QG} on $[0,\infty)$ with $\th(0) \in L^2$ and $f \in L^{p}$ for some $p>2$. Then $\th(t)$ satisfies the following energy equality:
\begin{equation}\notag
\frac{1}{2}\|\theta(t)\|_2^2+\nu\int_{t_0}^{t}\|\Lambda^{\frac{1}{2}}\theta(x,s)\|_2^2\,ds= 
\frac{1}{2}\|\theta(t_0)\|_2^2+\int_{t_0}^{t}\int_{\mathbb T^2}f\theta \, dxds,
\end{equation}
for all $0\leq t_0 \leq t$.
\end{Theorem}
\begin{proof}
Since $\th(t)$ is a viscosity solution,
\[
\lim_{t \to 0+} \|\th(t)\|_2 = \|\th(0)\|_2.
\]
Thus, it is enough to show that the energy flux vanishes on $[t_0,T]$ for every $0<t_0<T$. That is,
\[
\limsup_{Q \to \infty} \int_{t_0}^T\int_{\mathbb T^2}u\th \cdot\nabla (\theta_{\leq Q})_{\leq Q} \, dxdt=0.
\]
We use the Littlewood-Paley decomposition. As in \cite{CCFS},
denote 
\[r_Q(u,\theta)=\int_{\mathbb T^2}h_Q(y)\left(u(x-y)-u(x)\right) \left(\theta(x-y)-\theta(x)\right)dy.\]
Then 
\[(u\theta)_{\leq Q}=r_Q(u,\theta)-u_{>Q}\theta_{>Q}+u_{\leq Q}\theta_{\leq Q}.\]
The energy flux can be written as 
\begin{equation}\notag
\begin{split}
\Pi_Q=&\int_{\mathbb T^2}(u\theta)\cdot \nabla(\theta_{\leq Q})_{\leq Q}\,dx\\
=&\int_{\mathbb T^2}(u\theta)_{\leq Q}\cdot\nabla\theta_{\leq Q}\,dx\\
=&\int_{\mathbb T^2}r_Q(u,\theta)\cdot \nabla\theta_{\leq Q}dx
-\int_{\mathbb T^2}u_{>Q}\theta_{>Q}\cdot \nabla\theta_{\leq Q}\,dx.
\end{split}
\end{equation}
Here we used the fact that $\int_{\mathbb T^2}u_{\leq Q}\theta_{\leq Q}\cdot \nabla(\theta_{\leq Q})\,dx=0$ since $u$ is divergence free.

Now notice that
\begin{equation}\notag
\begin{split}
&\|u(\cdot-y)-u(\cdot)\|_2 \lesssim\sum_{p\leq Q}|y|\lambda_p\|u_p\|_2+\sum_{p> Q}\|u_p\|_2.\\
\end{split}
\end{equation}
Thanks to Lemma~\ref{Linfty}, there exists $C$, such that
$\|\theta(t)\|_\infty \leq C$ for all  $t >t_0$. Then

\begin{equation}\notag
\begin{split}
\|r_Q(u,\theta)\|_2&\leq \int_{\mathbb T^2}h_Q(y)\|u(\cdot-y)-u(\cdot)\|_2\|\theta(\cdot-y)-\theta(\cdot)\|_\infty\,dy\\
&\lesssim \int_{\mathbb T^2}h_Q(y)\left( \sum_{p\leq Q}|y|\lambda_p\|u_p\|_2
+\sum_{p> Q}\|u_p\|_2 \right)\, dy\\
&\lesssim \sum_{p\leq Q}\lambda_Q^{-1}\lambda_p\|u_p\|_2
+\sum_{p> Q}\|u_p\|_2.
\end{split}
\end{equation}
Therefore, applying Bernstein's inequality we obtain that
\begin{equation}\notag
\begin{split}
\left|\int_{\mathbb T^2}r_Q(u,\theta)\cdot\nabla\theta_{\leq Q}\,dx\right|
\leq& \|r_Q(u,\theta)\|_2\|\nabla \theta_{\leq Q}\|_2\\
\lesssim& \sum_{p\leq Q}\lambda_Q^{-1}\lambda_p\|u_p\|_2\left(\sum_{p'\leq Q}\lambda^2_{p'}\|\theta_{p'}\|_2^2\right)^{\frac 12}\\
&+\sum_{p> Q}\|u_p\|_2\left(\sum_{p'\leq Q}\lambda^2_{p'}\|\theta_{p'}\|^2_2\right)^{\frac 12}\\
:= &I+II.
\end{split}
\end{equation}
The terms $I$ and $II$ can be estimated using H\"older's inequality, Young's inequality and the fact $\|u_p\|_2\lesssim \|\theta_p\|_2$ as follows:
\begin{equation}\notag
\begin{split}
I\lesssim &\sum_{p\leq Q}\lambda_Q^{-1}\lambda_p\|u_p\|_2\left(\sum_{p'\leq Q}\lambda^2_{p'}\|\theta_{p'}\|_2^2\right)^{\frac 12}\\
\lesssim &\sum_{p\leq Q}\lambda_{p-Q}^{\frac 12}\lambda_p^{\frac 12}\|u_p\|_2\left(\sum_{p'\leq Q}\lambda_{p'-Q}\lambda_{p'}\|\theta_{p'}\|_2^2\right)^{\frac 12}\\
\lesssim &\sum_{p\leq Q}\lambda_{p-Q}^{\frac 12}\lambda_p\|\theta_p\|_2^2
+\sum_{p'\leq Q}\lambda_{p'-Q}\lambda_{p'}\|\theta_{p'}\|_2^2\\
\lesssim &\sum_{p\leq Q}\lambda_{p-Q}^{\frac 12}\lambda_p\|\theta_p\|_2^2;
\end{split}
\end{equation}
and
\begin{equation}\notag
\begin{split}
II\lesssim &\sum_{p> Q}\|u_p\|_2\left(\sum_{p'\leq Q}\lambda^2_{p'}\|\theta_{p'}\|_2^2\right)^{\frac 12}\\
\lesssim &\sum_{p> Q}\lambda_{Q-p}^{\frac 12}\lambda_p^{\frac 12}\|u_p\|_2\left(\sum_{p'\leq Q}\lambda_{p'-Q}\lambda_{p'}\|\theta_{p'}\|_2^2\right)^{\frac 12}\\
\lesssim &\sum_{p> Q}\lambda_{Q-p}^{\frac 12}\lambda_p\|\theta_p\|_2^2
+\sum_{p'\leq Q}\lambda_{p'-Q}\lambda_{p'}\|\theta_{p'}\|_2^2.
\end{split}
\end{equation}
On the other hand, we have a similar estimate
\begin{equation}\notag
\begin{split}
\left|\int_{\mathbb T^2}u_{>Q}\theta_{>Q}\cdot\nabla\theta_{\leq Q}\,dx\right|
\lesssim &\sum_{p>Q}\|\theta_p\|_2\left(\sum_{p'\leq Q}\lambda_{p'}^2\|\theta_{p'}\|_2^2\right)^{\frac 12}\\
\lesssim &\sum_{p> Q}\lambda_{Q-p}^{\frac 12}\lambda_p\|\theta_p\|_2^2
+\sum_{p'\leq Q}\lambda_{p'-Q}\lambda_{p'}\|\theta_{p'}\|_2^2.
\end{split}
\end{equation}
Therefore 
\begin{equation}\label{flux}
\left|\int_{t_0}^{T}\Pi_Q\,dt\right|\lesssim \sum_{p\leq Q}\lambda^{-\frac 12}_{|p-Q|}\int_{t_0}^{T}\lambda_p\|\theta_p\|_2^2\,dt.
\end{equation}
Since $\theta\in L_{loc}^2(0,\infty;H^{1/2}(\mathbb T^2))$, we have that
\[\int_{t_0}^{T}\lambda_p\|\theta_p\|_2^2\,dt\to 0 \qquad \mbox { as } p\to\infty.\]
It implies the right hand side of (\ref{flux}) converges to 0 as $Q\to\infty$, which gives that 
\[
\lim_{Q\to \infty}\int_{t_0}^{T} \Pi_Q \,dt=0.
\]
This completes the proof of the theorem. 
\end{proof}


\bigskip

\section{Global attractor}
\label{sec:att}

Let $(X,\ds(\cdot,\cdot))$ be a metric space endowed with
a metric $\ds$, which will be referred to as a strong metric. For the SQG equation, we will choose $X$ to be an absorbing ball, and $\ds$ be
the $L^2$-metric. 
Let $\dw(\cdot, \cdot)$ be another metric on $X$ satisfying
the following conditions:
\begin{enumerate}
\item $X$ is $\dw$-compact.
\item If $\ds(u_n, v_n) \to 0$ as $n \to \infty$ for some
$u_n, v_n \in X$, then $\dw(u_n, v_n) \to 0$ as $n \to \infty$.
\end{enumerate}
Due to the property 2, $\dw(\cdot,\cdot)$ will be referred to as a weak metric on $X$. Denote by $\overline{A}^{\bullet}$ the closure of a set $A\subset X$
in the topology generated by $\db$.
Note that any strongly compact ($\ds$-compact) set is weakly compact
($\dw$-compact), and any weakly closed set is strongly closed.

Let $C([a, b];X_\bullet)$, where $\bullet = \mathrm{s}$ or $\mathrm{w}$, be the space of $\db$-continuous $X$-valued
functions on $[a, b]$ 
endowed with the metric
\[
\dd_{C([a, b];X_\bullet)}(u,v) := \sup_{t\in[a,b]}\db(u(t),v(t)). 
\]
Let also $C([a, \infty);X_\bullet)$ be the space of $\db$-continuous
$X$-valued functions on $[a, \infty)$
endowed with the metric
\[
\dd_{C([a, \infty);X_\bullet)}(u,v) := \sum_{T\in \mathbb{N}} \frac{1}{2^T} \frac{\sup\{\db(u(t),v(t)):a\leq t\leq a+T\}}
{1+\sup\{\db(u(t),v(t)):a\leq t\leq a+T\}}.
\]

To define an evolutionary system, first let
\[
\mathcal{T} := \{ I: \ I=[T,\infty) \subset \mathbb{R}, \mbox{ or } 
I=(-\infty, \infty) \},
\]
and for each $I \subset \mathcal{T}$, let $\mathcal{F}(I)$ denote
the set of all $X$-valued functions on $I$.
\begin{Definition} \label{Dc}
A map $\Ec$ that associates to each $I\in \mathcal{T}$ a subset
$\Ec(I) \subset \mathcal{F}$ will be called an evolutionary system if
the following conditions are satisfied:
\begin{enumerate}
\item $\Ec([0,\infty)) \ne \emptyset$.
\item
$\Ec(I+s)=\{u(\cdot): \ u(\cdot +s) \in \Ec(I) \}$ for
all $s \in \mathbb{R}$.
\item $\{u(\cdot)|_{I_2} : u(\cdot) \in \Ec(I_1)\}
\subset \Ec(I_2)$ for all
pairs $I_1,I_2 \in \mathcal{T}$, such that $I_2 \subset I_1$.
\item
$\Ec((-\infty , \infty)) = \{u(\cdot) : \ u(\cdot)|_{[T,\infty)}
\in \Ec([T, \infty)) \ \forall T \in \mathbb{R} \}.$
\end{enumerate}
\end{Definition}
We will refer to $\Ec(I)$ as the set of all trajectories
on the time interval $I$. Trajectories in $\Ec((-\infty,\infty))$ will be called complete.
Let $P(X)$ be the set of all subsets of $X$.
For every $t \geq 0$, define a map
\begin{eqnarray*}
&R(t):P(X) \to P(X),&\\
&R(t)A := \{u(t): u\in A, u(\cdot) \in \Ec([0,\infty))\}, \qquad
A \subset X.&
\end{eqnarray*}
Note that the assumptions on $\Ec$ imply that $R(s)$ enjoys
the following property:
\begin{equation} \label{eq:propR(T)}
R(t+s)A \subset R(t)R(s)A, \qquad A \subset X,\quad t,s \geq 0.
\end{equation}

\begin{Definition}
A set $A \subset X$ is a $\mathrm{d}_{\bullet}$-attracting set
($\bullet = \mathrm{s,w}$) if it uniformly
attracts $X$ in $\mathrm{d}_{\bullet}$-metric.
\end{Definition}

\begin{Definition}
A set
$\mathcal{A}_{\bullet}\subset X$ is a
$\mathrm{d}_{\bullet}$-global attractor ($\bullet = \mathrm{s,w}$) if
$\mathcal{A}_{\bullet}$ is a minimal $\mathrm{d}_{\bullet}$-closed
$\mathrm{d}_{\bullet}$-attracting  set.
\end{Definition}

As we will see later, the evolutionary system $\Ec$ of whose trajectories are solutions to the
SQG equation also satisfies the following properties:
\begin{itemize}
\item[A1] $\Ec([0,\infty))$ is a compact set in $C([0,\infty); \Xw)$.
\item[A2] (Energy inequality) Assume that $X$ is a set in some Banach space $H$ satisfying the Radon-Riesz
property with the norm denoted by $\| \cdot \|$, so that $\ds(x, y) = \|x-y\|$ for $x, y \in X$ and $\dw$ induces the weak topology on X.
Assume also that for any $\epsilon >0$, there exists $\delta$, such that
for every $u \in \Ec([0,\infty))$ and $t>0$,
\[
\|u(t)\| \leq \|u(t_0)\| + \epsilon,
\]
for $t_0$ a.e. in $(t-\delta, t)\cap[0,\infty)$.
\item[A3] (Strong convergence a.e.) Let $u_n \in \Ec([0,\infty))$, be such that
$u_n \to u\in\Ec([0,\infty))$ in $C([0, T];\Xw)$ for some $T>0$. Then
$u_n(t) \to u(t)$ strongly a.e. in $[0,T]$.
\end{itemize}

\begin{Definition}
A Banach space $H$ with the norm $\|\cdot\|$ satisfies the Radon-Riesz property
if $x_n \to x$ in $H$  if and only if $x_n \to x$ weakly and $\lim \|x_n\| = \|x\|$ as $n \to \infty$.
\end{Definition}

The following results were proved in \cite{C5}:
\begin{Theorem} \label{thm:Attractor}
Let $\Ec$ be an evolutionary system satisfying $A1$. Then
the weak global attractor $\Aw$ exists and
\[
\Aw=\{u_0: u_0 =u(0) \mbox { for some } u\in \Ec((-\infty,\infty))\}.
\]
Furthermore, if $\Ec$ also satisfies A2, A3, and every complete trajectory is strongly continuous, then
\begin{enumerate}
\item The strong global attractor $\As$ exists, it is strongly compact, and $\As = \Aw$.
\item (Strong uniform tracking property) for any $\epsilon > 0$ and $T > 0$, there exists $t_0$, such that for any $t^* > t_0$, every trajectory
 $u\in\Ec([0,\infty))$ satisfies $\ds(u(t), v(t)) < \epsilon$, for all $t \in [t^*, t^* +T ]$, for some complete trajectory $v \in \Ec ((-\infty, \infty))$.
\end{enumerate}
\end{Theorem}

Now let us introduce some notations and functional setting.
Recall that $(\cdot,\cdot)$ and $\|\cdot\|_2$ are the $L^2$-inner product and the
corresponding $L^2$-norm.
Also, define the strong and weak distances by
\[
\ds(u,v):=\|u-v\|_2, \qquad
\dw(u,v)= \sum_{\nu \in \mathbb{Z}^3} \frac{1}{2^{|\nu|}}
\frac{|\hat{u}_{\nu}-\hat{v}_{\nu}|}{1 + |\hat{u}_{\nu}-\hat{v}_{\nu}|},
\qquad u,v \in L^2,
\]
where $\hat{u}_{\nu}$ and $\hat{v}_{\nu}$ are Fourier coefficients of $u$
and $v$ respectively.

Recall that the force $f$ has zero mean and $f \in L^{p}$ for some $p>2$. Thanks to the energy equality (Theorem~\ref{basic-energy}), 
there exists an absorbing ball for the 2D SQG
\[
X= \{\theta\in L^2(\mathbb{T}^2): \|\th\|_2 \leq R\},
\]
where $R$ as any number larger than $\|f\|_{H^{-1/2}(\mathbb{T}^2)} \nu^{-1} \sqrt{L/(2\pi)}$, where $L$ is the length of the periodic box.
Then for any bounded set $B \subset L^2$
there exists a time $t_0$, such that
\[
\th(t) \in X, \qquad \forall t\geq t_0,
\]
for every viscosity solution $\theta(t)$ with the initial data $\theta(0) \in B$.

Consider an evolutionary system for which
a family of trajectories consists
of all viscosity solutions of the SQG
in $X$. More precisely, define
\[
\begin{split}
\Ec([T,\infty)) := \{&\th(\cdot): \theta(\cdot)
\mbox{ is a viscosity solution of \eqref{QG} on } [T,\infty)\\
&\mbox{and } \th(t) \in X \ \forall t \in [T,\infty)\},
\qquad T \in \mathbb{R},
\end{split}
\]
\[
\begin{split}
\Ec((\infty,\infty)) := \{&\th(\cdot): \theta(\cdot)
\mbox{ is a viscosity solution of \eqref{QG} on } (-\infty,\infty)\\
&\mbox{and } \th(t) \in X \ \forall t \in (-\infty,\infty)\}.
\end{split}
\]

Clearly, the properties 1--4 of $\Ec$ hold.
Therefore, thanks to Theorem~\ref{thm:Attractor}, the weak global attractor $\Aw$
for this evolutionary system exists. Moreover, we have the following.
\begin{Lemma} \label{l:convergenceofLH}
Let $\th_n(t)$ be a sequence of viscosity solutions of the SQG,
such that $\th_n(t) \in X$ for all $t\geq t_0$. Then 
there exists a subsequence $\th_{n_j}$ of $\th_n$ that converges
in $C([t_0, T]; \Xw)$ to some viscosity solution $\th(t)$. 
\end{Lemma}
\begin{Lemma} \label{l:compact}
The evolutionary system $\Ec$ of the SQG
satisfies A1, A2, and A3.
\end{Lemma}
\begin{proof}
First note that $\Ec([0,\infty)) \subset C([0,\infty);\Xw)$. Now take any sequence
$\th_n \in \Ec([0,\infty))$, $n=1,2, \dots$.
Thanks to Lemma~\ref{l:convergenceofLH}, there exists
a subsequence, still denoted by  $\th_n$, that converges
to some $\th^{1} \in \Ec([0,\infty))$ in $C([0, 1];L^2_w)$ as $n \to \infty$.
Passing to a subsequence and dropping a subindex once more, we obtain that
$\th_n \to \theta^2$ in $C([0, 2];L^2_w)$ as $n \to \infty$ for some 
$\th^{2} \in \Ec([0,\infty))$.
Note that $\th^1(t)=\th^2(t)$ on $[0, 1]$.
Continuing
this diagonalization process, we obtain a subsequence $\th_{n_j}$
of $\th_n$ that converges
to some $\th \in \Ec([0,\infty))$ in $C([0, \infty);L^2_w)$ as $n_j \to \infty$.
Therefore, A1 holds.

Now, given $\epsilon>0$, let $\delta=\epsilon/(2 \|f\|_2R)$. Take any $\th \in \Ec([0,\infty))$
and $t>0$. Due to Theorem~\ref{basic-energy}, $\th(t)$ satisfies the energy equality
\[
\|\th(t)\|_2^2 + 2\nu \int_{t_0}^t \|\La^{1/2}\th(s)\|_2^2 \, ds =
\|\th(t_0)\|_2^2 + 2\int_{t_0}^t (f, \th(s)) \, ds,
\]
for all $0 \leq t_0 \leq t$. Hence,
\[
\begin{split}
\|\th(t)\|_2^2 &\leq \|\th(t_0)\|_2^2 + 2(t-t_0) \|f\|_2 R\\
&\leq \|\th(t_0)\|_2 + \epsilon,
\end{split}
\]
for all $t_0\geq 0$ such that $t_0 \in (t-\delta,t)$. Therefore, A2 holds.

Let now $\th_n \in \Ec([0,\infty))$ be such that $\th_n \to \th\in\Ec([0,\infty))$ in
$C([0, T];\Xw)$ as $n\to \infty$ for some
$T>0$. Thanks to the energy equality, the sequence $\{\th_n\}$
is bounded in $L^2([0,T];H^{\frac12})$. Hence,
\[
\int_{0}^T \|\th_n(s)-\th(s)\|_2^2 \, ds \to 0, \qquad \mbox{as}
\qquad  n \to \infty.
\]
In particular, $\|\th_n(t)\|_2 \to \|\th(t)\|_2$ as $n \to \infty$ a.e. on $[0,T]$,
i.e., A3 holds.
\end{proof}

Note that every solution of the SQG is strongly continuous due to Theorem~\ref{basic-energy}.
Now Theorem~\ref{thm:Attractor} yields the following.
\begin{Theorem}
The SQG equation possesses  a strongly compact strong global attractor $\A$,
\[
\A=\{ \theta_0: \ \theta_0=\theta(0) \mbox{ for some } \theta \in \Ec((-\infty, \infty))\}.
\] 
In addition, for any $\epsilon >0$ and $T>0$, there exists $t_0$,
such that for any $t^*>t_0$,  every viscosity solution $\theta \in \Ec([0,\infty))$
satisfies
\[
\ds(\theta(t), v(t)) < \epsilon, \qquad \forall t\in [t^*,t^*+T],
\]
for some complete trajectory $v \in \Ec((-\infty,\infty))$.
\end{Theorem}


{}

\end{document}